\documentclass[10pt]{amsart}
\usepackage{amssymb}
\usepackage{amsmath}
\usepackage{graphicx}
\usepackage{axodraw}
\usepackage{amsthm}
\usepackage{amscd}
\usepackage[all]{xy}               
\usepackage{epsfig,exscale}
\usepackage{color}

 \font \eightrm=cmr8

 \newcommand{\nc}{\newcommand}

\newtheorem{thm}{Theorem}
\newtheorem{exam}{Example}

\newtheorem{cor}[thm]{Corollary}

\newtheorem{lem}[thm]{Lemma}
\newtheorem{prop}[thm]{Proposition}

\newtheorem{rmk}[thm]{Remark}

\newcommand{\tree}{\scalebox{0.3}{{\parbox{0.5pc}{
  \begin{picture}(30,45) (75,-60)
    \SetWidth{1.5}
    \SetColor{Black}
    \Line(90,-30)(75,-60)
    \Line(90,-30)(105,-60)
    \Line(90,-15)(90,-30)
  \end{picture}}}}}

\newcommand{\treeC}{\scalebox{0.2}{{\parbox{0.5pc}{
 \begin{picture}(90,105) (75,-30)
    \SetWidth{2.4}
    \SetColor{Black}
    \Line(90,0)(75,-30)
    \Line(120,60)(90,0)
    \Line(90,0)(105,-30)
    \Line(105,30)(135,-30)
    \Line(120,60)(165,-30)
    \Line(120,75)(120,60)
  \end{picture}
}}}}

\newcommand{\treeD}{\scalebox{0.2}{{\parbox{0.5pc}{
 \begin{picture}(90,105) (75,-30)
    \SetWidth{2.4}
    \SetColor{Black}
    \Line(90,0)(75,-30)
    \Line(120,60)(90,0)
    \Line(90,0)(105,-30)
    \Line(120,60)(165,-30)
    \Line(120,75)(120,60)
    \Line(150,0)(135,-30)
  \end{picture}
}}}}

\newcommand{\treeF}{\scalebox{0.2}{{\parbox{0.5pc}{
 \begin{picture}(90,105) (75,-30)
    \SetWidth{2.4}
    \SetColor{Black}
    \Line(90,0)(75,-30)
    \Line(120,60)(90,0)
    \Line(120,60)(165,-30)
    \Line(120,75)(120,60)
    \Line(135,30)(105,-30)
    \Line(120,0)(135,-30)
  \end{picture}
}}}}

\newcommand{\treeG}{\scalebox{0.2}{{\parbox{0.5pc}{
 \begin{picture}(90,105) (75,-30)
    \SetWidth{2.4}
    \SetColor{Black}
    \Line(90,0)(75,-30)
    \Line(120,60)(90,0)
    \Line(120,60)(165,-30)
    \Line(120,75)(120,60)
    \Line(105,30)(135,-30)
    \Line(120,0)(105,-30)
  \end{picture}
}}}}

\nc{\ignore}[1]{{}}
\nc{\mrm}[1]{{\rm #1}}
\nc{\dirlim}{\displaystyle{\lim_{\longrightarrow}}\,}
\nc{\invlim}{\displaystyle{\lim_{\longleftarrow}}\,}
\nc{\vep}{\varepsilon} \nc{\ep}{\epsilon}
\nc{\sigmat}{\widetilde\sigma}
\nc{\ostar}{\overline{*}}

\nc{\mchar}{\mrm{Char}}
\nc{\Hom}{\mrm{Hom}}
\nc{\id}{\mrm{id}}

\nc{\remark}{\noindent{\bf{Remark:}}}
\nc{\remarks}{\noindent{\bf{Remarks:}}}

 \nc{\delete}[1]{}
 \nc{\grad}[1]{^{({#1})}}
 \nc{\fil}[1]{_{#1}}

\nc{\BA}{{\Bbb A}} \nc{\CC}{{\Bbb C}} \nc{\DD}{{\Bbb D}}
\nc{\EE}{{\Bbb E}} \nc{\FF}{{\Bbb F}} \nc{\GG}{{\Bbb G}}
\nc{\HH}{{\Bbb H}} \nc{\LL}{{\Bbb L}} \nc{\NN}{{\Bbb N}}
\nc{\PP}{{\Bbb P}} \nc{\QQ}{{\Bbb Q}} \nc{\RR}{{\Bbb R}}
\nc{\TT}{{\Bbb T}} \nc{\VV}{{\Bbb V}} \nc{\ZZ}{{\Bbb Z}}
\nc{\Cal}[1]{{\mathcal {#1}}}
\nc{\mop}[1]{\mathop{\hbox {\rm #1} }}
\nc{\smop}[1]{\mathop{\hbox {\eightrm #1} }}
\nc{\mopl}[1]{\mathop{\hbox {\rm #1} }\limits}
\nc{\frakg}{{\frak g}}
\nc{\g}[1]{{\frak {#1}}}
\def \restr#1{\mathstrut_{\textstyle |}\raise-8pt\hbox{$\scriptstyle #1$}}
\def \srestr#1{\mathstrut_{\scriptstyle |}\hbox to
  -1.5pt{}\raise-4pt\hbox{$\scriptscriptstyle #1$}}
\nc{\wt}{\widetilde}
\nc{\wh}{\widehat}
\nc{\un}{\hbox{\bf 1}}
\nc{\redtext}[1]{\textcolor{red}{#1}}
\nc{\bluetext}[1]{\textcolor{blue}{#1}}

\nc\fleche[1]{\mathop{\hbox to #1 mm{\rightarrowfill}}\limits}
\def\semi{\mathrel{\times}\kern -.85pt\joinrel\mathrel{\raise 1.4pt\hbox{${\scriptscriptstyle |}$}}}

\begin{document}

\title[A Magnus- and Fer-type formula in dendriform algebras]
      {A Magnus- and Fer-type formula in dendriform algebras}

\author{Kurusch Ebrahimi-Fard}
\address{Max-Planck-Institut f\"ur Mathematik,
         Vivatsgasse 7,
         D-53111 Bonn, Germany.}
 \email{kurusch@mpim-bonn.mpg.de}
 \urladdr{http://www.th.physik.uni-bonn.de/th/People/fard/}

\author{Dominique Manchon}
\address{Universit\'e Blaise Pascal,
         C.N.R.S.-UMR 6620,
         63177 Aubi\`ere, France}
         \email{manchon@math.univ-bpclermont.fr}
         \urladdr{http://math.univ-bpclermont.fr/~manchon/}

\date{December, 2007\\
\noindent {\footnotesize{${}\phantom{a}$ Mathematics Subject
Classification 2000: Primary: 16W25, 17A30, 17D25, 37C10 Secondary: 05C05, 81T15 \\ 
keywords: linear differential equation; linear integral equation; Magnus expansion; Fer expansion; dendriform algebra; 
pre-Lie algebra; Rota--Baxter algebra; binary rooted trees. }}
}

\begin{abstract}
We provide a refined approach to the classical
Magnus~\cite{Magnus} and Fer expansion~\cite{Fer}, unveiling a new
structure by using the language of dendriform and pre-Lie
algebras. The recursive formula for the logarithm of the solutions
of the equations $X = 1 + \lambda a \prec X$ and $Y = 1 - \lambda
Y \succ a$ in $A[[\lambda]]$ is provided, where $(A,\prec,\succ)$
is a dendriform algebra. Then, we present the solutions to these
equations as an infinite product expansion of exponentials. Both
formulae involve the pre-Lie product naturally associated with the
dendriform structure. Several applications are presented.
\end{abstract}

\maketitle

\tableofcontents


\section{Introduction}
\label{sect:intro}

Let us start by emphasizing that the results presented in the
sections following this introduction are an extension of findings
obtained by the authors together with Fr\'ed\'eric~Patras in an
earlier work~\cite{EMP07b}. The underlying theme of our paper is
to bring together Magnus', Fer's and Baxter's classical work on
linear differential equations respectively the corresponding
integral equations, using the language of Loday's dendriform
algebras. Skipping details, in this introduction we try to sketch
briefly the general picture behind our results.

Recall that Magnus~\cite{Magnus} and Fer~\cite{Fer}, as well as
Baxter~\cite{Baxter}, start their papers by recalling the
classical initial value problem:
\begin{equation}
\label{IVP1}
   \dot{\Phi}(t) := \frac{d}{dt}\Phi(t) = \Psi(t)\Phi(t),\qquad\ \Phi(0)=1.
\end{equation}

Magnus considers it in a non-commutative context, i.e. $A:=\Psi$
and $Y:=\Phi$ are supposed to be linear operators depending on a
real variable $t$. Here, $1$ denotes the identity operator. For
the linear operator $\Omega(t)$ depending on $A$ and with
$\Omega(0)=0$, such that:
$$
    Y(t) = \exp\Bigl( \int^t_0 \dot{\Omega}(s)\, ds \Bigr)
           = \sum_{n \ge 0} \frac{\Omega(t)^n}{n!},
$$
Magnus obtains a differential equation leading to the recursively
defined expansion named after him:
\begin{equation}
\label{MagnusOmega}
    \Omega(t) = \int^t_0 \dot{\Omega}(s)\, ds
              = \int_0^tA(s)\, ds + \int^t_0 \sum_{n > 0} \frac{B_n}{n!} ad_{\int_0^s \dot{\Omega}(u)\,du}^{(n)}[A(s)]\, ds.
\end{equation}
Here, as usual, $ad_f[g]:=fg-gf:=[f,g]$. The coefficients $B_n$
are the Bernoulli numbers defined via the generating series:
$$
    \frac{z}{\exp(z)-1} = \sum_{m \ge 0} \frac{B_m}{m!}z^m
                        = 0 - \frac{1}{2}z + \frac{1}{12}z^2 - \frac{1}{720}z^4 + \cdots.
$$
Observe that $B_{2m+3}=0$, $m \geq 0$ and recall for later use:
$$
    \frac{\exp(z)-1}{z}=\int^1_0 \exp(sz) \, ds.
$$
For more details see for
instance~\cite{Magnus,Gelfand,Iserles00,KlOt,MielPleb,OteoRos,Strichartz,Wilcox}.
Observe also that Magnus' expansion (\ref{MagnusOmega}) reduces to
$\int_0^t A(s)\, ds$ if all commutators disappear, e.g. in a
commutative setting, leading to the classical exponential
solution:
\begin{equation}
\label{clexpSol}
    Y(t)=\exp\Bigl(\int_0^t A(s)\, ds\Bigr)
\end{equation}
for the initial value problem (\ref{IVP1}). Indeed, the reason for
this exponential solution is simply encoded in the integration by
parts rule for $I(A)(t):=\int_0^tA(s)\,ds$:
\begin{equation*}
    \bigl(I(A)(t)\bigr)^n = n!\underbrace{I\Bigl(A\, I\bigl(A \cdots I(A)\cdots \bigr)\Bigr)}_{n\mbox{\rm -times}}(t).
\end{equation*}

Fer's~\cite{Fer} approach to solve the classical initial value
problem~(\ref{IVP1}), which was rediscovered by
Iserles~\cite{Iserles84} and further explored by
Zanna~\cite{Zanna}, and Munthe-Kaas and Zanna~\cite{MZ},
see~\cite{IsNo99} for more details, is somewhat different. His
Ansatz is:
$$
    Y=\exp\Bigl( \int_0^t A(s)\, ds \Bigr) V(t),
$$
which leads to the following differential equation for the
--correction-- operator $V(t)$:
$$
    \dot{V}(t)=\Bigl( \sum_{k>0}\frac{(-1)^k k}{(k+1)!} ad^{(k)}_{\int_0^t A(s)\, ds}[A(t)]
    \Bigr)V(t),\quad V(0)=1.
$$
Hence, an iteration of this method leads to the {\sl{Fer
expansion}}:
\begin{equation}
\label{FerExpProd}
    Y(t)=\exp\Bigl(\int_0^tU'_0(s)\,ds\Bigr)
         \exp\Bigl(\int_0^tU'_1(s)\,ds\Bigr)
         \exp\Bigl(\int_0^tU'_2(s)\,ds\Bigr)
         \cdots
         \exp\Bigl(\int_0^tU'_n(s)\,ds\Bigr)
         \cdots,
\end{equation}
where we use {\sl{Fer's recursion}}:
\begin{equation}
\label{FerRecu}
    U'_{m+1}(t):= \Bigl( \sum_{k>0}\frac{(-1)^k k}{(k+1)!} ad^{(k)}_{\int_0^t U'_m(s)\, ds}[U'_m(t)]
    \Bigr),\quad U'_0(t):=A(t).
\end{equation}

Taking an algebro-combinatorial perspective on these methods, we
should underline at this point that we completely skip the
analytical and numerical aspects of these expansions, which are
beyond doubt of crucial importance in applications. For this
purpose and related aspects we refer the interested reader to the
aforementioned references, e.g.
see~\cite{BCOR98,Iserles00,OteoRos}. However, at the end of this
work we report on an observation which may be of interest in this
context and which we plan to further explore in the near future.

\medskip

Baxter~\cite{Baxter} considers (\ref{IVP1}) in a commutative
setting, i.e. for continuous scalar functions $a:=\Psi$ and
$y:=\Phi$ depending on $t$. However, his starting point is the
corresponding integral equation:
\begin{equation}
\label{IVP2}
    y(t) = 1 + \int_0^t a(s)y(s)\, ds,
\end{equation}
and its exponential solution. Slightly deviating from Baxter's
original approach we generalize (\ref{IVP2}) to a formal power
series ring, $W[[\lambda]]$, where $W$ is a commutative unital
algebra over a field $k$ with a $k$-linear map $R: W \to W$
replacing the integral map:
\begin{equation}
\label{IVP3}
    Y = 1 + \lambda R(aY),
\end{equation}
$a \in W$ fixed. Here, $1$ is the unit in the algebra $W$. The map $R$ is
supposed to satisfy the relation:
\begin{equation}
\label{RBR}
    R(x)R(y) = R\bigl( R(x)y+xR(y) + \theta xy \bigr),
\end{equation}
where the parameter $\theta$ is a fixed scalar in $k$, called the
weight of $R$. One may think of (\ref{RBR}) as a generalized
integration by parts identity. Indeed, the reader will have no
difficulty in checking duality of~(\ref{RBR}) with the
`skewderivation' rule:
\begin{equation*}
    \partial(fg) = \partial(f)g + f\partial(g) + \theta \partial(f)\partial(g).
\end{equation*}
For example, the finite difference operator of step~$-\theta$,
given by $\partial f(x):=\theta^{-1}(f(x-\theta)-f(x))$, is a
skewderivation. On a suitable class of
functions, the summation operator:
\begin{equation}
\label{summation}
	S(f)(x) := \sum_{n\geq 1} \theta f(x + \theta n).
\end{equation}
satisfies relation (\ref{RBR}). Moreover:
\allowdisplaybreaks{
\begin{eqnarray*}
S\partial(f)(x) &=& \sum_{n\geq 1} \theta\partial(f)(x + \theta n) = \sum_{n\geq
1}\theta\,\frac{f(x +\theta n - \theta) - f(x + \theta n)}{\theta}
\\
&=& \sum_{n\geq 1} f\bigl(x +\theta (n - 1)\bigr) - f(x + \theta
n) = \sum_{n\geq 0} f(x +\theta n) - \sum_{n\geq 1}f(x + \theta n) = f(x).
\end{eqnarray*}}
And as the operator $\partial$ is linear we find as well $\partial S(f)=f$.  
Obviously, the skewderivation rule reduces to
Leibniz' rule for $\theta=0$, e.g. see~\cite{Kalliope} for more
details.

The exponential solution of equation (\ref{IVP3}) in $W[[\lambda]]$ can be seen as a
natural generalization of the classical exponential
solution~(\ref{clexpSol}) taking into account the weighted term in
identity~(\ref{RBR}):
\begin{equation}
\label{Spitzer}
    Y = 1 + \sum_{n>0} \lambda^n \underbrace{R\Bigl(a R\bigl(a \cdots R(a)}_{n-times} \bigl) \cdots  \Bigr)
      = \exp\Bigl(R\bigl(\frac{ \log(1 + \theta a \lambda)}{\theta}\bigr)\Bigr).
\end{equation}
The second equality is generally known as {\sl{Spitzer's classical
identity}}, see \cite{atkinson,Baxter,egm,Rota}. In fact, expanding 
the logarithm and the exponential on the right hand side, it
follows by comparing order by order in the parameter $\lambda$ the infinite 
set of identities in $W[[\lambda]]$.

Remember that the Riemann integral map $I:=\int_0^t$ satisfies
identity~(\ref{RBR}) for the weight $\theta=0$ (integration by
parts). In this particular case, observe that:
$$
	\theta^{-1} \log(1  + \theta a \lambda )= - \sum_{n>0} \frac{(-\theta)^{n-1}}{n}(a\lambda)^n \xrightarrow{\theta\downarrow 0}a \lambda.
$$ 
Hence, the exponential on the righthand side of identity~(\ref{Spitzer}) reduces to the classical
exponential solution~(\ref{clexpSol}). 

Having Magnus' work in mind it seems natural to ask for a non-commutative 
version of Spitzer's classical identity. In fact,  (\ref{Spitzer}) is only 
true when the underlying algebra is commutative, i.e. when (\ref{RBR}) implies:  
$$
	R(a)^2 = 2R(aR(a)) + \theta R(a^2).
$$
Its generalization to arbitrary weight $\theta$ 
{\sl{Rota--Baxter algebras}}, i.e. non-commutative algebras with a map 
$R$ satisfying relation (\ref{RBR}) can be found in~\cite{egm}. Similarly 
to Magnus' expansion~(\ref{MagnusOmega}), it relies on a particular
non-linear recursion $\chi_\theta$, which we call weight $\theta$
BCH-recursion as it is based on the Baker--Campbell--Hausdorff
formula:
\begin{equation}
\label{BCHchi}
    \chi_\theta(a) = a +
    \frac{1}{\theta}\mop{BCH}\bigl(\theta a,\, R\circ
    \chi_\theta(a)\bigr).
\end{equation}
We then find the non-commutative Spitzer identity: 
\begin{equation}
\label{SpitzerNC}
    Y = 1 + \sum_{n>0} \lambda^n \underbrace{R\Bigl(a R\bigl(a \cdots R(a)}_{n-times} \bigl) \cdots  \Bigr)
      = \exp\biggl(R\Bigl( \chi_{\theta}\bigl( \frac{ \log(1 + \theta a \lambda)}{\theta}\big)\Bigr)\biggr).
\end{equation}
We refer the reader to Section~\ref{sect:appl} for more details. One may also consult~\cite{egm,Kalliope}.

At first sight, the $\theta$ BCH- and Magnus recursion,
(\ref{BCHchi}) respectively (\ref{MagnusOmega}), look different,
but it is the goal of this work to show how they are related.
Indeed, in~\cite{egm} it was already shown that (\ref{BCHchi})
reduces to Magnus' formula (\ref{MagnusOmega}) in the
corresponding limit $\theta \to 0$. Here, on the contrary, we will
show how to get (\ref{BCHchi}) from the limit case $\theta=0$.
After this has been achieved we provide a refined picture of Fer's
expansion~(\ref{FerRecu}) in purely algebraic terms leading to an
infinite product expansion of the solution to (\ref{IVP3}). These
results are achieved in the context of the generalized integration
by parts rule~(\ref{RBR}). We emphasize the use of Loday's
dendriform algebras~\cite{Loday} which appears to be well-suited.

Recall that a dendriform algebra is an --associative-- algebra
with two non-associative operations, written $\prec$ and $\succ$
satisfying three rules. The two products add to form the product
of the algebra. At the same time they define a left and right
pre-Lie product on the same algebra. J.-L. Loday recently
introduced this notion in connection with dialgebra
structures~\cite{Loday,LodayRonco}. The commutative version of dendriform algebra is called 
``dual Leibniz algebra'' or ``Zinbiel algebra'' by Loday \cite{Loday2} (see also~\cite{Sch58}). 
Temporarily, this algebra respectively its product had been called ``chronological algebra'' 
respectively  ``chronological product'' in the context of control theory, e.g. see~\cite{Kaw00} 
for more details. The key point from our perspective is the
intimate relation between associative algebras, equipped with a
map satisfying relation (\ref{RBR}) and such dendriform algebras.
We will see that this connection renders dendriform algebras a
suitable setting to encode algebraic structures related to the
integral equations corresponding to linear differential equations.
At the end we will indicate that this approach unveils a hitherto
hidden structure in Fer's and Magnus' original work leading to a
reduction in number of commutator terms. The reader may also look
into~\cite{Kalliope} for related issues.

\medskip

This work is organized as follows. In Section~\ref{sect:dpse} we
recall the notion of dendriform algebra and introduce two
particular dendriform power sums expansions.
Section~\ref{sect:preLieMag} contains the first main result of
this work, that is, we present the central object, a new pre-Lie
Magnus type recursion, and show that its exponentiation solves the
aforementioned pair of dendriform power sums expansions. Then, in
Section~\ref{sect:preLieFer}, we introduce the pre-Lie Fer
recursion as the second main result, leading to an infinite
product expansion of exponentials for these solutions. Finally, in
Section~\ref{sect:appl} we provide some applications. We finish
the article with an observation indicating hitherto overlooked new
structural properties of the classical Magnus and Fer expansions
due to the extra pre-Lie relation, which leads to a reduction in
the number of terms in these expansions.


\section{Dendriform power sums expansions}
\label{sect:dpse}

Let $k$ be a field of characteristic zero. Recall that a {\sl
dendriform algebra\/}~\cite{Loday} over $k$ is a $k$-vector
space $A$ endowed with two bilinear operations $\prec$ and $\succ$
subject to the three axioms below:
\begin{eqnarray}
 (a\prec b)\prec c &=& a\prec(b*c)        \label{A1}\\
 (a\succ b)\prec c &=& a\succ(b\prec c)   \label{A2}\\
  a\succ(b\succ c) &=& (a*b)\succ c        \label{A3}.
\end{eqnarray}
In the commutative case, the left and right operations are
further required to identify, so that $x \succ y = y \prec x$. 
One can show that these relations yield associativity for the
product 
\begin{equation}
\label{dend-assoc}
	a * b := a \prec b + a \succ b.  
\end{equation}

\begin{exam}\label{ex:RiemannDend}{\rm{
As a guiding example we regard an algebra $F$ of operator-valued functions on 
the real line, closed under integrals $\int_0^x$, say, smooth $n \times n$ matrix-valued
functions. Then, $D_F=(F,\prec,\succ)$ is a dendriform algebra for the operations:
$$
    (A \prec B) (x) := A(x)\cdot \int\limits_0^x B(y)\, dy
 \qquad
    (A \succ B) (x) := \int\limits_0^x A(y)\, dy \cdot B(x)
$$
with $A,B \in F$. One verifies the dendriform axioms using the integration by parts rule. 
For instance, in this setting the dendriform relation (\ref{A1}) means:
$$
\Big(A(x)\cdot \int\limits_0^x B(u)\, du \Big)\cdot \int\limits_0^x C(v)\, dv 
= A(x)\cdot \int\limits_0^x  \biggl( B(u)\cdot \int\limits_0^u C(v)\, dv\ +\ \int\limits_0^u B(v)\, dv \cdot C(u)\biggr)\, du
$$ 
for $A,B,C \in F$. The associative product in this dendriform algebra then writes:
$$
    (A * B)(x) := A(x)\cdot \int\limits_0^x B(y)\, dy\ +\ \int\limits_0^x A(y)\, dy \cdot B(x).
$$
Let us remark that a commutative algebra $(F,\int_0^x)$ naturally provides a commutative dendriform algebra. 
}}
\end{exam}

Let us return to the dendriform axioms. One shows that at the same time the dendriform relations 
imply that the bilinear operations
$\rhd$ and $\lhd$ defined by:
\begin{equation}
\label{def:prelie}
    a \rhd b:= a\succ b-b\prec a,
    \hskip 12mm
    a \lhd b:= a\prec b-b\succ a
\end{equation}
are {\sl{left pre-Lie}} and {\sl{right pre-Lie}}, respectively,
which means that we have:
\begin{eqnarray}
    (a\rhd b)\rhd c-a\rhd(b\rhd c)&=& (b\rhd a)\rhd c-b\rhd(a\rhd c),\label{prelie1}\\
    (a\lhd b)\lhd c-a\lhd(b\lhd c)&=& (a\lhd c)\lhd b-a\lhd(c\lhd b).\label{prelie2}
\end{eqnarray}

In the setting of our guiding Example \ref{ex:RiemannDend} these pre-Lie products write:
$$
    (A \lhd B)(x) := A(x) \cdot \int\limits_0^x B(y)\, dy  \ -\  \int\limits_0^x B(y)\, dy \cdot A(x),
$$
$$
    (A \rhd B)(x) := \int\limits_0^x A(y)\, dy \cdot B(x) \ - \ B(x) \cdot \int\limits_0^x A(y)\, dy.
$$

The associative operation $*$ and the pre-Lie operations $\rhd$,
$\lhd$ all define the same Lie bracket:
\begin{equation}
    [a,b]:=a*b-b*a=a\rhd b-b\rhd a=a\lhd b-b\lhd a.
\end{equation}

Loday and Ronco introduced in~\cite{LodayRonco} the notion of
{\textsl{tridendriform algebra}} $T$ equipped with three operations, $<, >$
and $\bullet$, satisfying seven dendriform type axioms:
 \allowdisplaybreaks{
\begin{eqnarray}
    && (x < y) < z = x < (y \star z),   \hskip 7mm
       (x > y) < z = x > (y < z),       \hskip 7mm
       (x \star y) > z = x > (y > z),   \\
    (x > y) \bullet z \!\!\!&=&\!\!\! x > (y \bullet z),   \hskip 3mm
       (x < y) \bullet z = x \bullet (y> z),    \hskip 3mm
       (x \bullet y) < z = x \bullet (y < z),   \hskip 3mm
       (x \bullet y) \bullet z = x \bullet (y \bullet z), \notag
       \label{DT}
\end{eqnarray}}
yielding an associative product $x \star y := x < y + x > y + x
\bullet y$. First, observe that the category of dendriform
algebras can be identified with the subcategory of objects in
the category of tridendriform algebras with $\bullet=0$. Moreover,
one readily verifies for a tridendriform algebra $(T,<,>,\bullet)$
that $(D_T,\prec_\bullet,\succ)$, where $\prec_\bullet :=\ < +\
\bullet$ and $\succ:=>$, is a dendriform algebra~\cite {E}.

\begin{exam}\label{ex:SumDend}{\rm{ The summation operation 
(\ref{summation}) for $\theta=1$ on a suitable algebra $F$ of functions provides a natural 
example for such a tridendriform algebra $T_F=(F,<,>,\bullet)$:
$$
	(A < B)(x):= 	A(x) \cdot S(B)(x),
	\;\
	(A > B)(x):=  S(A)(x) \cdot B(x),
	\;\
	(A \bullet B)(x):= A(x) \cdot B(x).
$$
One verifies without problems the tridendriform relations. For instance, the first relation simply encodes:
$$
	\Big(A(x) \cdot S(B)(x)\Big) \cdot S(C)(x) = A(x) \cdot \Big( S\big(B\cdot S(C)\big)(x) 
											+ S\big(S(B)\cdot C\big)(x) 
											+ S(B \cdot C)(x)\Big)
$$
identity (\ref{RBR}). The associative product in this tridendriform algebra then writes:
$$
    (A \star B)(x) = A(x)\cdot S(B)(x)\ +\  S(A)(x) \cdot B(x) + A(x) \cdot B(x).
$$
Within this example the reader may want to convince himself that 
$D_{T_F}=(F,\prec_\bullet,\succ)$, where $\prec_\bullet :=\ < +\ \bullet$, i.e.:
$$
	(A \prec_\bullet B)(x)= (A < B)(x) + (A \bullet B)(x)
$$  
and $\succ:=>$, is a dendriform algebra. 
}}
\end{exam}

One feels that dendriform algebras provide an elegant setting for a refined encoding of fundamental structures
underlying integration and summation operations. In fact, further below we will show that this is true for a much 
larger class of operators, i.e. characterized by identity (\ref{RBR}).

\medskip

Let $\overline A = A \oplus k.\un$ be our dendriform algebra
augmented by a unit $\un$:
\begin{equation}
\label{unit-dend}
    a \prec \un := a =: \un \succ a
    \hskip 12mm
    \un \prec a := 0 =: a \succ \un,
\end{equation}
implying $a*\un=\un*a=a$. Note that $\un*\un=\un$, but that $\un
\prec \un$ and $\un \succ \un$ are not defined~\cite{R}, \cite{C}.
We recursively define the following set of elements of $\overline
A[[\lambda]]$ for a fixed $a \in A$:
 \allowdisplaybreaks{
\begin{eqnarray*}
    w^{(0)}_{\prec}(a) := \un =:w^{(0)}_{\succ}(a),\ \;
    w^{(n)}_{\prec}(a) := a \prec \bigl(w^{(n-1)}_\prec(a)\bigr),\ \;
    w^{(n)}_{\succ}(a) := \bigl(w^{(n-1)}_\succ(a)\bigr)\succ a.
\end{eqnarray*}}

Let us define the exponential and logarithm map in terms of the
associative product (\ref{dend-assoc}), $\exp^*(x):=\sum_{n \geq 0} x^{*n}/n!$,
$\log^*(\un+x):=-\sum_{n>0}(-1)^nx^{*n}/n$, respectively. In the
following we first give a recursive expression for the logarithm
of the solutions of the following two equations for a fixed $a \in
A$:
\begin{equation}
\label{eq:prelie}
     X = \un + \lambda a \prec X,
    \hskip 12mm
     Y = \un - Y \succ \lambda a.
\end{equation}
in $\overline A[[\lambda ]]$, in terms of the left pre-Lie product
$\rhd$. This will in particular encompass the Magnus expansion for
the logarithm of a solution of a linear first-order homogeneous
differential equation in a noncommutative algebra~\cite{Magnus}.
Then we present the solutions of the two equations as an infinite
product expansion of the exponential, which encompasses Fer's
solution to a linear first-order homogeneous differential
equation~\cite{Fer}.

\begin{rmk}
{\rm{ For later use we may also write:
$$
    (X-\un) = \lambda a + \lambda a \prec (X-\un)
    \hskip 15mm {\rm{resp.}} \hskip 15mm
    (Y-\un) = - \lambda a - (Y-\un) \succ \lambda a.
$$
}}
\end{rmk}


\section{The pre-Lie Magnus expansion}
\label{sect:preLieMag}

Formal solutions to (\ref{eq:prelie}) are given by:
\begin{equation*}
    X = \sum_{n \geq 0} \lambda^nw^{(n)}_{\prec}(a)
    \hskip 15mm {\rm{resp.}} \hskip 15mm
    Y = \sum_{n \geq 0} (-\lambda)^nw^{(n)}_{\succ}(a).
\end{equation*}
Let us introduce the following operators in $(A,\prec,\succ)$,
where $a$ is any element of $A$:
 \allowdisplaybreaks{
\begin{eqnarray*}
    L_\prec[a](b)&:=& a \prec b \hskip 8mm L_\succ[a](b):= a\succ b \hskip 8mm
    R_\prec[a](b):= b \prec a \hskip 8mm R_\succ[a](b):= b\succ a\\
    L_\lhd[a](b)&:=& a \lhd b \hskip 8mm L_\rhd[a](b) := a \rhd b \hskip 8mm
    R_\lhd[a](b):=b \lhd a \hskip 8mm R_\rhd[a](b):= b \rhd a
\end{eqnarray*}}

\begin{thm} \label{thm:main}
Let $\Omega':=\Omega'(\lambda a)$, $a \in A$, be the element of
$\lambda \overline{A}[[\lambda]]$ such that $X=\exp^*(\Omega')$
and $Y=\exp^*(-\Omega')$, where $X$ and $Y$ are the solutions of
the two equations (\ref{eq:prelie}), respectively. This element
obeys the following recursive equation:
 \allowdisplaybreaks{
\begin{eqnarray}
    \Omega'(\lambda a) &=& \frac{R_\lhd[\Omega']}{1-\exp(-R_\lhd[\Omega'])}(\lambda a)=
                \sum_{m \ge 0}(-1)^m \frac{B_m}{m!}R_\lhd[\Omega']^m(\lambda a),\label{main1}
\end{eqnarray}}
or alternatively:
 \allowdisplaybreaks{
\begin{eqnarray}
    \Omega'(\lambda a) &=& \frac{L_\rhd[\Omega']}{\exp(L_\rhd[\Omega'])-1}(\lambda a)
            =\sum_{m\ge 0}\frac{B_m}{m!}L_\rhd[\Omega']^m(\lambda a),\label{main3}
\end{eqnarray}}
where the $B_l$'s are the Bernoulli numbers.
\end{thm}

\begin{proof}
Let us notice that (\ref{main3}) can be immediately derived from
(\ref{main1}) thanks to $L_\rhd[b]=-R_\lhd[b]$ for any $b\in A$.
We prove (\ref{main1}), which can be rewritten as:
\begin{equation}
    \lambda a = \frac{1 - \exp(-R_\lhd[\Omega'])}{R_\lhd[\Omega']}(\Omega'(\lambda a)).
\label{main5}
\end{equation}
Given such $\Omega':=\Omega'(\lambda a) \in \lambda
\overline{A}[[\lambda]]$ we must then prove that
$X:=\exp^*(\Omega'(\lambda a))$ is the solution of $X = \un +
\lambda a \prec X$, where $a$ is given by (\ref{main5}). Let us
first remark that:
\begin{equation}
    R_\lhd[\Omega'] = R_\prec[\Omega'] - L_\succ[\Omega'],
\end{equation}
and that the two operators $R_\prec[\Omega']$ and
$L_\succ[\Omega']$ commute thanks to the dendriform axiom
(\ref{A2}). We have then, using the three dendriform algebra
axioms:
 \allowdisplaybreaks{
\begin{eqnarray*}
    \lambda a =  \frac{1-\exp(-R_\lhd[\Omega'])}{R_\lhd[\Omega']}(\Omega')
      &=& \int_{0}^1 \exp(-sR_\lhd[\Omega'])(\Omega')\,ds\\
      &=& \int_{0}^1 \exp(sL_\succ[\Omega'])\exp(-sR_\prec[\Omega'])(\Omega')\,ds\\
      &=& \int_{0}^1 \exp^*(s\Omega') \succ \Omega' \prec \exp^*(-s\Omega')\, ds.
\end{eqnarray*}}
So we get:
 \allowdisplaybreaks{
\begin{eqnarray}
    \lambda a \prec X &=& \int_0^1 \bigl(\exp^*(s\Omega') \succ \Omega' \prec \exp^*(-s\Omega') \bigr)\prec \exp^*(\Omega')\, ds\nonumber\\
               &=& \int_0^1 \exp^*(s\Omega') \succ \Omega' \prec \exp^*((1-s)\Omega')\, ds\label{eq:important}\\
               &=& \sum_{n\ge 0} \sum_{p+q=n} \Omega'^{*p} \succ \Omega' \prec \Omega'^{*q}
                                              \int_0^1 \frac{(1-s)^q s^p}{p!q!}\, ds.\nonumber
\end{eqnarray}}
An iterated integration by parts shows that:
\begin{equation*}
    \int_0^1 (1-s)^q s^p\,ds = \frac{p!q!}{(p+q+1)!},
\end{equation*}
which yields:
\begin{equation*}
    \lambda a \prec X = \sum_{n\ge 0}\frac{1}{(n+1)!}\sum_{p+q=n}\Omega'^{*p}\succ \Omega' \prec \Omega'^{*q}.
\end{equation*}
On the other hand, we have:
\begin{equation}
    X-\un=\exp^*(\Omega')-\un =\sum_{n\ge 0} \frac{1}{(n+1)!} \Omega'^{*n+1}.
\end{equation}
Equality (\ref{main5}) follows then from the identity:
\begin{equation*}
    \sum_{p+q=n}\Omega'^{*p} \succ \Omega' \prec \Omega'^{*q} =\Omega'^{*n+1}
\end{equation*}
which is easily shown by induction on $n$. Analogously, one
readily verifies that:
 \allowdisplaybreaks{
\begin{eqnarray*}
    Y \succ \lambda a &=& \int_0^1  \exp^*(-\Omega') \succ \bigl(\exp^*(s\Omega') \succ \Omega' \prec \exp^*(-s\Omega') \bigr)\, ds\\
               &=& \int_0^1 \exp^*((s-1)\Omega') \succ \Omega' \prec \exp^*(-s\Omega')\, ds\\
               &=& \sum_{n\ge 0} \sum_{p+q=n} (-1)^{(p+q)}\Omega'^{*p} \succ \Omega' \prec \Omega'^{*q}
                                              \int_0^1 \frac{(1-s)^q s^p}{p!q!}\, ds
                = \sum_{n\ge 0}\frac{(-1)^{(p+q)}}{(n+1)!}\sum_{p+q=n}\Omega'^{*p}\succ \Omega' \prec \Omega'^{*q}.
\end{eqnarray*}}

\end{proof}

\begin{rmk}\label{OmegaPrim}{\rm{
It seems appropriate at this point to justify notation. We have
chosen to write $\Omega'$ to remind the reader of the fact that in
the particular context of Example~\ref{ex:RiemannDend} (see
subsection~\ref{subsect:clMagnus} below) one readily sees that
$\Omega'(t) = \dot{\Omega}(t)$, see eq.~(\ref{MagnusOmega}).}}
\end{rmk}

\begin{rmk}\label{OmegaDend}{\rm{In~\cite{EMP07b} we were able to show, using
Hopf algebra and free Lie algebra techniques, i.e. the Dynkin
idempotent map that:
\begin{equation*}
    \Omega'(a) = \int_0^1 \biggl(\mathcal{L}(s) +
    \sum_{n>0} (-1)^n\frac{B_n}{n!} ad^{(n)}_{\Omega'}
    (\mathcal{L}(s))\biggr)ds,
\end{equation*}
where $\mathcal{L}(t)=\sum_{n>0}R_\rhd[a]^{(n)} (a)t^{n-1}$. }}
\end{rmk}


\section{The pre-Lie Fer expansion}
\label{sect:preLieFer}

Let us come back to the dendriform power sums
expansions~(\ref{eq:prelie}):
$$
    X = \un + \lambda a \prec X
    \qquad\ \ \rm{ and }\ \qquad\
    Y= \un - \lambda Y \succ a.
$$
We will mainly focus on the first one. Following Fer's original
work~\cite{Fer} we now make a simple Ansatz for its solution:
$$
    X=\exp^*(\lambda a)*V_1
$$
and return this into the recursion. Recall the dendriform
axiom~(\ref{A1}). This then leads to the following:
 \allowdisplaybreaks{
\begin{eqnarray*}
    \exp^*(\lambda a)*V_1 &=& \un + \lambda a \prec \bigl(\exp^*( \lambda a) * V_1\bigr)
                           =  \un + \lambda \bigl(a \prec \exp^*(\lambda a)\bigr) \prec V_1.
\end{eqnarray*}}
It is the goal to derive a recursion for $V_1$ similar to the
original one and then to iterate the process.
Remember~(\ref{unit-dend}) and the dendriform axiom~(\ref{A2}):
 \allowdisplaybreaks{
\begin{eqnarray}
    V_1 &=& \exp^*(-\lambda a) + \lambda \exp^*(-\lambda a) * \Bigl( \bigl(a \prec \exp^*(\lambda a)\bigr) \prec V_1\Bigr) \nonumber\\
        &=& \exp^*(-\lambda a) + \lambda \exp^*(-\lambda a) \prec \Bigl( \bigl(a \prec \exp^*(\lambda a)\bigr) \prec V_1\Bigr) +
                                 \lambda \exp^*(-\lambda a) \succ \Bigl( \bigl(a \prec \exp^*(\lambda a)\bigr) \prec V_1\Bigr) \nonumber\\
        &=& \exp^*(-\lambda a) + \lambda \exp^*(-\lambda a) \prec \Bigl( \bigl(a \prec \exp^*(\lambda a)\bigr) \prec V_1\Bigr) +
                  \bigl(\exp^*(-\lambda a) \succ \lambda a \prec \exp^*(\lambda a)\bigr) \prec V_1 \nonumber\\
        &=& \un + \bigl(\exp^*(-\lambda a)-\un \bigr) \prec
                  \Bigl(\un + \lambda  a \prec \bigl( \exp^*(\lambda a) * V_1\bigr)\Bigr) +
                  \bigl(\exp^*(-\lambda a) \succ \lambda a \prec \exp^*(\lambda a)\bigr) \prec V_1 \nonumber\\
        &=& \un + \Bigl(\bigl(\exp^*(-\lambda a)-\un \bigr) \prec \exp^*(\lambda a) \Bigr) \prec V_1 +
                  \bigl(\exp^*(-\lambda a) \succ \lambda a \prec \exp^*(\lambda a)\bigr) \prec V_1 \nonumber\\
        &=& \un + \Bigl(\bigl(\exp^*(-\lambda a)-\un \bigr) \prec \exp^*(\lambda a) +
                  \exp^*(-\lambda a) \succ \lambda a \prec \exp^*(\lambda a) \Bigr) \prec V_1. \nonumber
\end{eqnarray}}
At this point we can repeat the above using the Ansatz:
$$
    V_1:=\exp^*( U'_1 )*V_2,
$$
where:
\begin{equation}
\label{here1}
    U'_1:= \bigl(\exp^*(-\lambda a)-\un \bigr) \prec \exp^*(\lambda a) +
                 \exp^*(-\lambda a) \succ \lambda a \prec \exp^*(\lambda a).
\end{equation}
In general we have:
$$
    V_{n} := \exp^*( U'_{n} )*V_{n+1},
$$
with $U'_0:=\lambda a$ and:
\begin{equation}
\label{here3}
    U'_{n}:= \bigl(\exp^*(-U'_{n-1})-\un \bigr) \prec \exp^*(U'_{n-1}) +
                 \exp^*(-U'_{n-1}) \succ U'_{n-1} \prec \exp^*(U'_{n-1}).
\end{equation}
Such that we arrive at the following infinite product expansion
for $X$:
$$
    X = \exp^*( U'_0 ) * \exp^*( U'_1 ) * \exp^*( U'_2 ) * \cdots * \exp^*( U'_n )* \cdots.
$$
Analogously, using the Ansatz $Y = V_1*\exp^*(-\lambda a)$ one
shows that:
$$
    Y = \cdots * \exp^*(- U'_n ) * \cdots * \exp^*( -U'_2 ) * \exp^*( -U'_1 ) * \exp^*( -U'_0 ).
$$

Let us now examine a bit closer~(\ref{here1}). By the foregoing
calculation in Section~\ref{sect:preLieMag} one readily verifies
that:
$$
   \exp(R_\lhd[\lambda a])(\lambda a) = \exp^*(-\lambda a) \succ \lambda a \prec \exp^*(\lambda a).
$$

\begin{lem} \label{keytoFer} Let $(A,\prec,\succ)$ be a dendriform algebra
augmented by a unit $\un$~(\ref{unit-dend}). Then, for $a \in A$
we have:
\begin{equation}
\label{here2}
    \bigl(\exp^*(-a)-\un \bigr) \prec \exp^*(a) = -  \int_{0}^1 \exp^*(-s a) \succ a \prec \exp^*(s a)\, ds.
\end{equation}
\end{lem}

\begin{proof}
In the proof of Theorem~\ref{thm:main} we have shown (see
(\ref{eq:important}), with $-a$ replacing $\Omega'$) that:
$$
    \exp^*(-a)-\un = -  \int_{0}^1 \exp^*(-s a) \succ a \prec \exp^*((s-1) a)\, ds.
$$
This immediately yields:
 \allowdisplaybreaks{
\begin{eqnarray*}
    \bigl(\exp^*(-a) -\un \bigr) \prec \exp^*(a)
    &=& - \Big(\int_{0}^1 \exp^*(-s a) \succ a \prec \exp^*((s-1) a)\, ds \Big) \prec \exp^*(a)\\
    &=& - \int_{0}^1 \exp^*(-s a) \succ a \prec \exp^*((s) a)\, ds
\end{eqnarray*}}
by application of the first dendriform axiom.
\end{proof}

This lemma implies then for the general recursion~(\ref{here3}):
 \allowdisplaybreaks{
\begin{eqnarray}
   \lefteqn{ \bigl(\exp^*(-U'_{n})-\un \bigr) \prec \exp^*(U'_{n}) +
                   \exp^*(-U'_{n}) \succ U'_{n} \prec \exp^*(U'_{n})} \nonumber \\
    &=& \exp(R_\lhd[U'_{n}])(U'_{n})  -  \int_{0}^1 \exp^*(s U'_{n}) \succ U'_{n} \prec \exp^*(-s U'_{n})\, ds \nonumber\\
    &=& \exp(R_\lhd[U'_{n}])(U'_{n})  - \frac{\exp(R_\lhd[U'_{n}])-1}{R_\lhd[U'_{n}]}(U'_{n}) \nonumber\\
    &=& \exp(-L_\rhd[U'_{n}])(U'_{n})  + \frac{\exp(-L_\rhd[U'_{n}])-1}{L_\rhd[U'_{n}]}(U'_{n}) \nonumber\\
    &=& \frac{\bigl(L_\rhd[U'_{n}] + 1\bigr)\exp(-L_\rhd[U'_{n}]) - 1}{L_\rhd[U'_{n}]}(U'_{n})\nonumber
\end{eqnarray}}
and one shows that this then gives the nice identity essentially
encoding Fer's classical expansion in terms of a pre-Lie product:
$$
    \bigl(\exp^*(-U'_{n})-\un \bigr) \prec \exp^*(U'_{n}) +
                   \exp^*(-U'_{n}) \succ U'_{n} \prec \exp^*(U'_{n})
    = \sum_{l>0} \frac{(-1)^l l}{(l+1)!} L_\rhd[U'_n]^{(l)}(U'_n).
$$

In the following theorem we summarize the above by formulating
Fer's expansion~(\ref{FerRecu}).

\begin{thm} \label{thm:Fer}
Let $(A,\prec,\succ)$ be a dendriform algebra augmented by a
unit $\un$~(\ref{unit-dend}). Let $U'_0:=\lambda a$,
$U'_n:=U'_n(a)$, $n \in \mathbb{N}$, $a \in A$, be elements in
$\lambda \overline{A}[[\lambda]]$, such that
$$
    X = \overrightarrow{\prod\limits_{n \ge 0}}^* \exp^*(U'_n)
    \qquad\
    Y = \overleftarrow{\prod\limits_{n \ge 0}}^* \exp^*(-U'_n)
$$
where $X$ and $Y$ are the solutions of the equations
(\ref{eq:prelie}). Then these elements $U'_n$ obey the following
recursive equation:
 \allowdisplaybreaks{
\begin{eqnarray}
    U'_{n+1}:= \sum_{l>0} \frac{(-1)^l l}{(l+1)!} L_\rhd[U'_n]^{(l)}(U'_n),\quad\ n \ge 0.
\end{eqnarray}}
\end{thm}

The presentation given here in the context of dendriform algebras
reduces to Fer's classical expansion when working in the
dendriform algebra of Example~\ref{ex:RiemannDend}. Let us recall
that Iserles~\cite{Iserles84} rediscovered Fer's result calling it
the method of iterated commutators. Munthe-Kaas and
Zanna~\cite{MZ} further developed Iserles' work in the context of
Lie group integrators.


\section{Applications}
\label{sect:appl}

In this section we collect some applications where the foregoing
straightforwardly implies known as well as new results.


\subsection{Associative algebras}
\label{subsect:assoAlg}

Any associative algebra $(A,*)$ can be seen as a dendriform
algebra with $\prec=*$ and $\succ=0$ (or alternatively $\prec=0$
and $\succ=*$). In this case the pre-Lie operation $\lhd$ reduces
to the associative product, and equation (\ref{main5}) reduces to:
\begin{equation}
    \lambda a=\un-\exp^*(-\Omega'),
\end{equation}
hence $\Omega'=-\log^* (\un-\lambda a)$. Its exponential
$X=(\un-\lambda a)^{*-1}=\un+\lambda a+\lambda^2a*a+\cdots$ indeed
verifies $X=\un+\lambda a*X$.


\subsection{Rota--Baxter algebras}
\label{subsect:RB}

Recall \cite{Baxter,E,Rota} that an associative Rota--Baxter
algebra (over a field $k$) is an associative $k$-algebra $A$
endowed with a $k$-linear map $R: A \to A$ subject to the
following relation:
\begin{equation}\label{RB}
    R(a)R(b) = R\bigl(R(a)b + aR(b) + \theta ab\bigr).
\end{equation}
where $\theta \in k$. The map $R$ is called a {\sl Rota--Baxter
operator of weight $\theta$\/}. The map $\widetilde{R}:=-\theta id
-R$ also is a weight $\theta$ Rota--Baxter map. Both the image of $R$ 
and $\tilde{R}$ form subalgebras in $A$. Associative
Rota--Baxter algebras arise in many mathematical contexts, e.g. in
integral and finite differences calculus, but also in perturbative
renormalization in quantum field theory~\cite{egm}.

A few examples are in order. On a suitable class of functions, we define the
following Riemann summation operators
\begin{eqnarray}
    R_\theta(f)(x) := \sum_{n = 1}^{[x/\theta]} \theta f(n\theta)
    \qquad\ {\rm{and}} \qquad\
    R'_\theta(f)(x) := \sum_{n = 1}^{[x/\theta]-1} \theta f(n\theta).
\label{eq:le-clou}
\end{eqnarray}
Observe that
\begin{align}
    &\biggl( \sum_{n = 1}^{[x/\theta]} \theta f(n\theta)\biggr)
     \biggl( \sum_{m = 1}^{[x/\theta]} \theta g(m\theta)\biggr)
   = \biggl( \sum_{n > m = 1}^{[x/\theta]}
        + \sum_{m > n = 1}^{[x/\theta]}
        + \sum_{m = n = 1}^{[x/\theta]}  \biggr)\theta^2 f(n\theta) g(m\theta) \nonumber \\
    &= \sum_{m = 1}^{[x/\theta]} \theta^2 \biggl(\sum_{k = 1}^{m} f\bigl(k\theta\bigr)\biggr)
                                                        g(m\theta)
     + \sum_{n = 1}^{[x/\theta]} \theta^2 \biggl(\sum_{k = 1}^{n} g\bigl(k\theta\bigr)\biggr)
                                                                             f(n\theta) \nonumber
     - \sum_{n = 1}^{[x/\theta]} \theta^2 f(n\theta)g(n\theta)\\
    &= R_\theta\bigl(R_\theta(f)g\bigr)(x) + R_\theta\bigl(fR_\theta(g)\bigr)(x) + \theta R_\theta(fg)(x).
\label{Riemsum1}
\end{align}
Similarly for the map $R'_\theta$. Hence, the Riemann summation
maps $R_\theta$ and $R'_\theta$ satisfy the weight $-\theta$ and
the weight $\theta$ Rota--Baxter relation, respectively. 

Let us give another example, very different from summation and integration maps.
Let $A$ be a $\mathbb{K}$-algebra which decomposes directly into subalgebras 
$A_1$ and $A_2$ , $A = A_1 \oplus A_2$, then the projection to $A_1$, $R: A \to A$,
$R(a_1,a_2)=a_1$, is an idempotent Rota--Baxter operator, i.e. of weight $\theta=-1$. 
Let us verify this for $a,b \in A= A_1 \oplus A_2$
\begin{eqnarray*}
    R(a)b + aR(b) - ab &=& R(a)\big(R(b) + (id-R)(b)\big) - \big(R(a) + (id-R)(a)\big)(id-R)(b)\\
                      &=&R(a)R(b) - (id-R)(a)(id-R)(b)
\end{eqnarray*}
such that applying $R$ on both sides kills the term $(id-R)(a)(id-R)(b)$  without changing the
term $R(a)R(b)$, as $R (id-R)(a)=0$ since $A_1,A_2$ are
subalgebras.

\begin{prop}\label{RBtridend}\cite {E}
Any associative Rota--Baxter algebra gives a {\sl{tridendriform
algebra}}, $(T_R,<,>,\bullet_\theta)$, in the sense that the
Rota--Baxter structure yields three binary operations:
 \allowdisplaybreaks{
\begin{eqnarray*}
    a < b := aR(b),
    \hskip 8mm
    a > b := R(a)b,
,   \hskip 8mm
    a \bullet_\theta b := \theta ab,
\end{eqnarray*}}
satisfying the tridendriform algebra axioms (\ref{DT}).
\end{prop}

The associated associative product $*_\theta$ is given by
 \allowdisplaybreaks{
\begin{eqnarray*}
\label{RBasso}
    a *_\theta b := aR(b) + R(a)b + \theta ab
\end{eqnarray*}}
It is sometimes called the ``double Rota--Baxter product'', and
verifies:
\begin{equation}
\label{RBhom}
    R( a *_\theta b) = R(a)R(b), \quad \widetilde{R}( a *_\theta b)=- \widetilde{R}(a)\widetilde{R}(b)
\end{equation}
which is just a reformulation of the Rota--Baxter relation
(\ref{RB}). It follows from the general link between dendriform and
tridendriform algebras that any Rota--Baxter algebra gives rise to a
dendriform algebra structure, $(D_R,\prec,\succ)$, given by:
 \allowdisplaybreaks{
\begin{eqnarray}\label{dendRB}
    a \prec b &:=& aR(b)+\theta ab =-a\widetilde{R}(b),\hskip 8mm a \succ b:=R(a)b.
\end{eqnarray}}
For completeness we mention the following.

\begin{rmk} Rota--Baxter Dendriform algebra {\rm{It is easy to
verify that $(D_R,\prec,\succ )$ defines a Rota--Baxter dendriform
algebra with weight $\theta$ Rota--Baxter map $R:D_R \to D_R$,
that is:
 \allowdisplaybreaks{
\begin{eqnarray*}
    \wt R(a)\prec \wt R(b) &=& \wt R(\wt R(a)\prec b + a\prec \wt R(b) +\theta a\prec b)\\
    R(a)\succ R(b) &=& R(R(a)\succ b + a\succ R(b) +\theta a\succ b)
\end{eqnarray*}}
This provides an example for a non-associative Rota--Baxter
algebra.}}
\end{rmk}

The (weight $\theta$) Rota--Baxter (left) pre-Lie operation
corresponding to (\ref{prelie1}) is given by:
\begin{eqnarray}
\label{RBpreLie}
    a \rhd b &=& R(a)b - bR(a) - \theta ba = [R(a),b] - \theta ba.
\end{eqnarray}
Let us remark that the underlying vector space $A$ equipped with
associative product $*_\theta$ is again a Rota--Baxter algebra
with weight $\theta$ Rota--Baxter map $R$. Whereas, the pre-Lie
algebra $(A,\rhd)$ is a Rota--Baxter pre-Lie algebra with weight
$\theta$ Rota--Baxter map $R$, i.e. another example for a
non-associative Rota--Baxter algebra. For completeness, recall the
known fact that, analogously to~(\ref{RBasso}), we can define a
new pre-Lie product $\rhd_\theta$ on $(A,\rhd)$, and
$(A,\rhd_\theta)$ is again a Rota--Baxter pre-Lie algebra with
weight $\theta$ Rota--Baxter map $R$.

Notice that if we suppose the algebra $A$ to be unital, the unit
(which we denote by $1$) has nothing to do with the artificially
added unit $\un$ of the underlying dendriform algebra. We extend
the Rota--Baxter algebra structure to $\overline A$ by setting:
\begin{equation}\label{unite}
    R(\un):=1,\hskip 8mm \wt R(\un):=-1 \hskip 5mm \hbox{ and
    }\un.x=x.\un=0\
    \hbox{ for any } x\in\overline A.
\end{equation}
This is consistent with the axioms (\ref{unit-dend}) which in
particular yield $\un\succ x=R(\un)x$ and $x\prec \un=-x\wt
R(\un)$, in coherence with (\ref{dendRB}).

Now let us suppose that the Rota--Baxter algebra $A$ is unital,
and introduce the  {\sl{weight $\theta \in k$ pre-Lie Magnus type
recursion}}, $\Omega'_\theta:=\Omega'_\theta(\lambda a) \in
\lambda \overline{A}[[\lambda]]$, where the Rota--Baxter operator
$R$ is naturally extended to $\overline A[[\lambda]]$ by
$k[[\lambda]]$-linearity:
 \allowdisplaybreaks{
\begin{eqnarray}
    \Omega'_\theta(\lambda a) &=& \sum_{m \ge 0} \frac{B_m}{m!}L_{\rhd}[\Omega'_\theta]^{(m)}(\lambda a),
\label{RB-MagunsMain}
\end{eqnarray}}

Theorem~\ref{thm:main} together with relation~(\ref{RBhom})
implies for a fixed $a \in A$ the following corollary.

\begin{cor}\label{cor:RBrecursions}
The elements $\hat{X}:=-\wt
R(X)=\exp\bigl(-\widetilde{R}(\Omega'_\theta(\lambda a))\bigr)$
and $\hat{Y}:=R(Y)=\exp\bigl(-R(\Omega'_\theta(\lambda a))\bigr)$
in $A[[\lambda]]$ solve the equations:
\begin{equation}
\label{RBrecursions}
    \hat{X} = 1 - \lambda \widetilde{R}(a\hat{X})
     \hskip 15mm {\rm{resp.}} \hskip 15mm
    \hat{Y} = 1 - \lambda R(\hat{Y}a).
\end{equation}
\end{cor}

\begin{proof}
Recall the link between dendriform and tridendriform algebras, and the
Rota--Baxter relation, i.e. relations~(\ref{RBtridend}).
Using~(\ref{RBhom}), we have:
 \allowdisplaybreaks{
\begin{eqnarray*}
    \hat{X} &=& -\wt R(X)\\
            &=& -\wt R(\un + \lambda a\prec X)
             =  -\wt R\big(\un - \lambda a\wt R(X)\big)
             =  -\wt R\big(\un + \lambda a \hat{X}\big)\\
            &=& 1 - \lambda \wt R(a\hat{X}).
\end{eqnarray*}}
and similarly:
\begin{eqnarray*}
    \hat{Y}  &=& R(Y)\\
            &=& R(\un - \lambda Y\succ a)=R(\un-\lambda R(Y)a)=R(\un-\lambda \hat{Y} a)\\
            &=& 1 - \lambda R(\hat{Y} a).
\end{eqnarray*}
\end{proof}

We may summarize the picture we developed so far in the following diagram relating the pre-Lie Magnus expansion
to the original Magnus expansion as well as to the Spitzer's identity:

\begin{equation*}
    \xymatrix{
         & {\hbox{{\eightrm{$\exp\!\bigg(\! R\Big(\sum\limits_{m\ge 0} \frac{B_m}{m!}\ L_{\bullet_R}[\Omega'_\theta]^{(m)}(\lambda a)\Big)\!\bigg)$}}}
                                                    \atop \theta \neq 0,\ non-com. }
                                                     \ar[dd]^{{com.\atop \theta \to 0}}
                                                     \ar[rd]_{\theta \neq 0 \atop com.}
                                                     \ar[ld]^{\theta \to 0 \atop {non-com.}}
                                                     &\\
   {\hbox{\bf{$\exp\!\big(\Omega_0(a)\!\big)$}}
    \atop {\rm{Magnus}}}
   \ar[rd]^{com.}\ar@<5pt>[ru]^{equation
         (\ref{conjecturebis})\ \ }  &
                                                    & {\hbox{{$\exp\!\bigg(\!R\Big(
                                                     \!\frac{\log(1+ \theta a \lambda)}{\theta}
                                                     \!\Big)\!\bigg)$}}
                                                      \atop {\rm{cl.\ Spitzer}}}
                                                      \ar[ld]_{\theta \to 0}\\
                   & {\hbox{\eightrm{$\exp\!\big(R(a)\big)$}}
                      \atop \theta = 0,\ com.}&
                }
\end{equation*}
On the left wing we regard the limit $\theta \to zero$ leading to Magnus' expansion, whereas on the left wing we observe the 
reduction of the pre-Lie Magnus expansion to the logarithm in Spitzer's identity when the underlying algebra is commutative. Finally
both formulas reduce to the classical exponential solution in a commutative weight zero Rota--Baxter algebra. We refer the reader 
to \cite{egm} for more details.

\medskip

We define here the {\sl{weight $\theta \in k$ pre-Lie Fer type
recursion}}, $U'_{n,\theta}:=U'_{n,\theta}(a) \in \lambda
\overline{A}[[\lambda]]$:
 \allowdisplaybreaks{
\begin{eqnarray}
    U'_{n+1,\theta}(\lambda a) &=& \sum_{l>0} \frac{(-1)^l l}{(l+1)!} L_\rhd[U'_{n,\theta}]^{(l)}(U'_{n,\theta}),\quad\ n \ge 0.
\label{RB-Ferexp}
\end{eqnarray}}

By the same line of arguments Theorem~\ref{thm:Fer} implies the
next corollary.

\begin{cor}\label{cor:RBrecursions2}
The element
$$
    \hat{X}:= \overrightarrow{\prod\limits_{n \ge 0}} \exp\Bigl(-\widetilde{R}\bigl(U'_{n,\theta}( a)\bigr)\Bigr)
    \qquad\
    \hat{Y}:= \overleftarrow{\prod\limits_{n \ge 0}} \exp\Bigl(- R\bigl(U'_{n,\theta}( a)\bigr)\Bigr)
$$
in $A[[\lambda]]$ solves the above recursions.
\end{cor}

Associative Rota--Baxter algebras are essentially characterized by
a natural factorization theorem related to the
equations~(\ref{RBrecursions}), see Atkinson~\cite{atkinson} for
more details.

Recall \cite{egm} that there is a unique (usually non-linear)
bijection $\chi_\theta:\lambda A[[\lambda]]\to \lambda
A[[\lambda]]$, which is a deformation of the identity, such that
for any $\alpha \in \lambda A[[\lambda]]$ we have\footnote{Notice
the sign change compared to \cite{egm}. This is due to a
conventional sign change in the definition of the weight: an
idempotent Rota--Baxter operator is now of weight $-1$ with the
new convention. The two recursive formulae for $\chi_\theta$ which
follow thus differ a little bit from those given in the
aforementioned reference.}:
\begin{equation*}
    \exp(-\theta\alpha) = \exp\bigl(R\circ \chi_\theta(\alpha)\bigr)\exp\bigl(\wt R\circ\chi_\theta(\alpha)\bigr).
\end{equation*}
The so-called $\theta$ BCH-recursion map $\chi_\theta$ is
recursively given by:
\begin{equation*}
    \chi_\theta(\alpha) =\alpha + \frac{1}{\theta}\mop{BCH}\bigl(R\circ\chi_\theta(\alpha),\, \wt R\circ\chi_\theta(\alpha) \bigr),
\end{equation*}
or alternatively:
\begin{equation*}
    \chi_\theta(\alpha) = \alpha +
    \frac{1}{\theta}\mop{BCH}\bigl(\theta\alpha,\, R\circ \chi_\theta(\alpha) \bigr),
\end{equation*}
where $\mop{BCH}(x,y)$ is the Baker--Campbell--Hausdorff series
defined by $\exp(x)\exp(y)=\exp\bigl(x+y+\mop{BCH}(x,y)\bigr)$.
Recall Atkinson's theorem \cite{atkinson}:
\begin{equation*}
    \hat Y(1 - \theta \lambda a)\hat X = 1,
\end{equation*}
hence $\hat Y^{-1}\hat X^{-1} = 1-\theta \lambda a$. We deduce
immediately from the very definitions of $\hat X$ and $\hat Y$
that we have:
\begin{equation}\label{rwtr}
    1-\theta \lambda a = \exp(-\theta\alpha_\theta)=\exp\bigl(R(\Omega'_\theta)\bigr)\exp\bigl(\wt R(\Omega'_\theta)\bigr),
\end{equation}
with $\alpha_\theta:=\alpha_\theta(\lambda a):=-\frac{1}{\theta}\log (1 - \theta \lambda a)$. We then infer from (\ref{rwtr}) the equality:
\begin{equation}
\label{conjecture}
    \Omega'_\theta = \Omega'_\theta(\lambda a)
           = \chi_{\theta}(\alpha_\theta)
           = \chi_\theta\Bigl(-\frac{\log (1-\theta \lambda a)}{\theta}\Bigr),
\end{equation}
which was conjectured in~\cite{Kalliope} in a similar context.
From~(\ref{conjecture}) we get for any $\alpha \in \lambda
A[[\lambda]]$:
\begin{equation}
\label{conjecturebis}
    \chi_\theta\big(\alpha) = \Omega'_\theta\Bigl(\frac{1-\exp(\theta\alpha)}{\theta}\Bigr).
\end{equation}


\subsection{The weight zero case and the classical Magnus expansion}
\label{subsect:clMagnus}

For a weight $\theta = 0$, Rota--Baxter algebra the pre-Lie
product~(\ref{RBpreLie}) reduces to:
\begin{eqnarray}
\label{zeroRBpreLie}
    a \rhd b&=& [R(a),b].
\end{eqnarray}
This simplifies the weighted pre-Lie Magnus type
recursion~(\ref{RB-MagunsMain}) to the {\sl{classical Magnus
recursion}}:
 \allowdisplaybreaks{
\begin{eqnarray}
    \Omega'_0(\lambda a) = \chi_0(\lambda a)
                &=& \sum_{m \ge 0} \frac{B_m}{m!}L_{\rhd}[\chi_0]^{(m)}(\lambda a) \notag\\
                &=& \sum_{m \ge 0} \frac{B_m}{m!}ad_{R(\chi_0)}^{(m)}(\lambda a).
\label{clMaguns}
\end{eqnarray}}

An important example for a weight zero Rota--Baxter algebra is any
algebra $F$ of operator-valued functions on the real line, closed
under integrals $\int_0^x$, say, smooth $n \times n$ matrix-valued
functions. Recall example (\ref{ex:RiemannDend}), showing that $F$ is a dendriform 
algebra. The associative product then writes:
$$
    A * B(x) := A(x)\cdot \int\limits_0^x B(y)\, dy + \int\limits_0^x A(y)\, dy \cdot B(x)
$$
such that:
$$
    \exp\Bigl( \int_0^x A(y)\, dy \Bigr) =  \int_0^x \bigl(\exp^*( A(y))\bigr) \, dy.
$$
Recall our convention (\ref{unite}) for the dendriform unit. Then the classical Magnus 
recursion~(\ref{MagnusOmega}), $\Omega(t)=\Omega(A)(t)$, for $R:=\int^t_0$ reads:
 \allowdisplaybreaks{
\begin{eqnarray*}
    \dot{\Omega}(A)(t)=\Omega'_0(A)(t) = \sum_{m \ge 0} \frac{B_m}{m!}L_{\rhd}[\dot{\Omega}(A)]^{(m)}(A)(t)
                      =\sum_{m \ge 0} \frac{B_m}{m!}ad_{\int^t_0 \dot{\Omega}(A)(s)\,ds}^{(m)}(A(t)).
\end{eqnarray*}}
Similarly, in the case of the weighted pre-Lie Fer expansion as
expected we recover the original formula. 

We showed in reference \cite{egm} that the recursion $\chi_\theta$
reduces to the classical Magnus recursion~(\ref{clMaguns}) in the
limit $\theta \to 0$:
$$
    \chi_\theta(\lambda a) \xrightarrow{\theta \to 0} \chi_0(\lambda a).
$$
We recover this result from (\ref{conjecturebis}),
(\ref{RB-MagunsMain}), and from the fact that
$\alpha_\theta(\lambda a)\to \lambda a$ when $\theta \to 0$.
Hence, we have proven that the weight $\theta$
BCH-recursion~$\chi_\theta$ derives naturally from the classical
Magnus recursion via the non-linear change of variable
$\alpha_\theta$.

\medskip

Let us finish this brief note with an interesting observation to
be further explored in a future work. In the context of
Example~\ref{ex:RiemannDend} and the initial value
problem~(\ref{IVP1}) we have seen that the classical Magnus and
Fer recursion can be rewritten using the (weight zero) pre-Lie
product introduced by the integral operator $\int_0^x$:
 \allowdisplaybreaks{
\begin{eqnarray*}
    \dot{\Omega}_0(A)(s)
             =  \sum_{m \ge 0} \frac{B_m}{m!}L_{\rhd}[\dot{\Omega}_0(A)(s)]^{(m)}\bigl(A(s)\bigr),
\end{eqnarray*}}
and
 \allowdisplaybreaks{
\begin{eqnarray*}
    \dot{U}_{n+1,0}(A)(s)
            = \sum_{l>0} \frac{(-1)^l l}{(l+1)!} L_\rhd[\dot{U}_{n,0}(A)(s)]^{(l)}\bigl(\dot{U}_{n,0}(A(s))\bigr).
\end{eqnarray*}}

\medskip

It seems to be natural to ask whether the extra pre-Lie structure
in the case of the Magnus as well as the Fer expansion implies a
possible reduction in the number of terms. Here we would like to
indicate that this is the case, by explicit verification up to
fifth order in the case of the Magnus expansion, using the
dendriform algebra respectively the induced pre-Lie product
presented in Example~\ref{ex:RiemannDend}. We should emphasize the
fact, that the following applies, of course, to the more general
weight $\theta$ pre-Lie Magnus and Fer type
recursions~(\ref{RB-MagunsMain}), (\ref{RB-Ferexp}), respectively.

We introduce a dummy parameter $\lambda$ for convenience.
Obviously, as the (left) pre-Lie relation~(\ref{prelie1}) is a
ternary one we expect it to be available only from third order
upwards in this parameter.
 \allowdisplaybreaks{
\begin{align}
      \Omega(t)=\int_0^t \dot{\Omega}_0(A\lambda)(s) \, ds
                     &=  \lambda            \int_0^t A(s)\, ds
                       - \lambda^2 \frac 12 \int_0^t A \rhd A(s)\, ds \\
                     & + \lambda^3          \int_0^t \Bigl(  \frac{1}{12}\bigl(A \rhd (A \rhd A)\bigr)(s)
                                                           + \frac{1}{4} \bigl((A \rhd A) \rhd A\bigr)(s)\Bigr)\, ds \nonumber \\
                     & + \lambda^4 \int_0^t \bigg(
                                       - \frac{1}{8}    \bigl((A \rhd A) \rhd  A \bigr) \rhd A(s)
                                       - \frac{1}{24} \bigl(A \rhd (A \rhd A)\bigr) \rhd A(s)                  \nonumber \\
                     & \hskip 16mm - \frac{1}{24} \Bigl(    A \rhd \bigl((A \rhd A) \rhd A \bigr)(s)
                                             + (A \rhd A) \rhd (A \rhd A)(s) \Bigr) \bigg)\, ds +
                                              \mathcal{O}(5)
                                              \nonumber
\end{align}}
Recall that $A \rhd A(s) =[\int_0^s A(u)\, du,\,\ A(s)]$. We see
that at third order no further reduction of terms is possible. At
fourth order we find a reduction to two terms using the pre-Lie
relation~(\ref{prelie1}). Indeed, one verifies that:
 \allowdisplaybreaks{
\begin{eqnarray*}
    \frac{1}{8} \bigl((A \rhd A) \rhd  A \bigr) \rhd A
  + \frac{1}{24} \Bigl( \bigl(A \rhd (A \rhd A)\bigr) \rhd A
                            + A \rhd \bigl((A \rhd A) \rhd A \bigr)
                            +(A \rhd A) \rhd (A \rhd A) \Bigr)
\end{eqnarray*}}
using that, thanks to the pre-Lie relation:
$$
    (A \rhd A) \rhd (A \rhd A) =
    \bigl((A \rhd A) \rhd  A \bigr) \rhd A
    - \bigl(A \rhd (A \rhd A)\bigr) \rhd A
    + A \rhd \bigl((A \rhd A) \rhd A \bigr)
$$
equals:
$$
 \frac{1}{6} \bigl((A \rhd A) \rhd  A \bigr) \rhd A
 + \frac{1}{12} A \rhd \bigl((A \rhd A) \rhd  A \bigr).
$$
At fifth order we observe a reduction in the number of terms from
ten to seven. More details and a complete analysis of this
apparently new structure in the Magnus (and Fer) expansion will be
provided in a forthcoming work.

Following the seminal work of Iserles and
N{\o}rsett~\cite{IsNo99}, using planar rooted binary trees to
encode the combinatorial setting in Magnus expansion, we may
present the foregoing calculation more transparently. The binary
tree:
$$
    \tree \quad \sim \ A \rhd A.
$$
At fourth order we have:
$$
    \frac{1}{8} \ \treeC  \hskip 10mm
    +  \hskip 5mm \frac{1}{24}\biggl(
    \hskip 3mm
    \treeG  \hskip 10mm +  \hskip 5mm
    \treeF  \hskip 10mm +  \hskip 5mm
    \treeD  \hskip 8mm
    \bigg)
$$
which reduces to:
$$
    \frac{1}{6}\ \treeC \hskip 10mm +  \hskip 2mm \frac{1}{12}\ \treeF
$$
thanks to the left pre-Lie relation:
$$
    \treeD \hskip 10mm - \hskip 5mm \treeC
           \hskip 10mm = \hskip 5mm
    \treeF \hskip 10mm - \hskip 5mm \treeG
$$

\vspace{0.5cm}

\textbf{Acknowledgements}

\smallskip

The first named author acknowledges greatly the support by the
European Post-Doctoral Institute. He thanks A.~Iserles,
H.~Munthe-Kaas and B.~Owren for helpful discussions, and the
Department of Applied Mathematics and Theoretical Physics at
Cambridge University for warm hospitality. A special thanks goes
to A.~V\"olzmann-Scheiding at MPIM Bonn. We would like to thank
Fr\'ed\'eric Fauvet for inviting us to Institut de Recherche
Math\'ematique Avanc\'ee, Universit\'e Louis Pasteur were part of
this work was done. We would like to thank Fr\'ed\'eric Patras and
Jose M.~Gracia-Bond\'{\i}a for helpful discussions, and Jean-Louis Loday for
useful comments.

\end{document}